\documentclass{amsart}
\usepackage[T1]{fontenc}
\usepackage[english]{babel}
\usepackage{amsmath,amsthm,amssymb, latexsym,geometry,stackrel,enumerate}
\usepackage[all]{xy}
\usepackage{hyperref}
\usepackage{graphicx}
\DeclareGraphicsExtensions{.bmp,.png,.pdf,.jpgg}
\usepackage[all]{xy}
\usepackage{verbatim}
\usepackage[mathcal]{euscript}
\usepackage{epsfig}
\usepackage{multicol}
\usepackage{color}
\usepackage{verbatim}
\usepackage{graphicx,float}

\usepackage{setspace}
\newcommand{\marginparstretch}{0.6}
\let\oldmarginpar\marginpar
\renewcommand\marginpar[1]{\-\oldmarginpar[\framebox{\setstretch{\marginparstretch}\begin{minipage}{\marginparwidth}{\raggedleft\tiny #1}\end{minipage}}]{\framebox{\setstretch{\marginparstretch}\begin{minipage}{\marginparwidth}{\raggedright\tiny #1}\end{minipage}}}}

\usepackage{adjustbox}

\date{\today}

\swapnumbers
\theoremstyle{plain}
\newtheorem{teo}{Theorem}[section]
\newtheorem{lema}[teo]{Lemma}
\newtheorem{prop}[teo]{Proposition}
\newtheorem{coro}[teo]{Corollary}

\theoremstyle{definition}
\newtheorem{defi}[teo]{Definition}
\newtheorem{obs}[teo]{Remark}

\newtheorem{ejem}[teo]{Example}

\def\Ext{\mathop{\rm Ext}\nolimits}

\newcommand{\repdim}{\operatorname{rep.dim.}\,}
\newcommand{\ind}{\operatorname{ind}}
\newcommand{\Gen}{\operatorname{Gen}}
\newcommand{\cogen}{\operatorname{Cogen}}
\renewcommand{\mod}{\operatorname{mod}}
\newcommand{\Y}{\operatorname{\mathcal{Y}}}
\renewcommand{\L}{\operatorname{\mathcal{L}}}
\newcommand{\R}{\operatorname{\mathcal{R}}}
\newcommand{\add}{\operatorname{add}}
\newcommand{\pd}{\operatorname{pd}}
\newcommand{\id}{\operatorname{id}}
\newcommand{\gldim}{\operatorname{gl.dim.}}

\newcommand{\Hom}{\operatorname{Hom}}
\newcommand{\End}{\operatorname{End}}
\newcommand{\Ann}{\operatorname{Ann}}
\newcommand{\findim}{\operatorname{fin.dim.}\,}
\newcommand{\wrepdim}{\operatorname{w.rep.dim.}\,}

\input{xy}
\xyoption{poly}
\xyoption{2cell}
\xyoption{all}

\usepackage{url}

\newcommand{\rep}[1]{%
  {%
    \tiny%
    \begin{matrix}%
      #1%
    \end{matrix}%
  }%
}

\usepackage{pgf,tikz}
\usepackage{mathrsfs}
\usetikzlibrary{arrows,decorations.pathmorphing,backgrounds,positioning,fit,petri,cd}

\title[Algebras determined by $\tau$-slices]{Algebras determined by $\tau$-slices}

\author[V. Gubitosi]{Viviana Gubitosi}
\address{Instituto de Matem\'{a}tica y Estad\'{\i}stica Rafael Laguardia, Facultad de Ingenier\'{\i}a - UdelaR, Montevideo, Uruguay, 11200 }
\email{gubitosi@fing.edu.uy}

\author[H. Treffinger]{Hipolito Treffinger}
\address{Universidad de Buenos Aires. Facultad de Ciencias Exactas y Naturales. Departamento de Matemática. Buenos Aires, Argentina.
CONICET - Universidad de Buenos Aires. Instituto de Investigaciones Matemáticas “Luis A. Santaló”  (IMAS). Buenos Aires, Argentina.}
\email{htreffinger@dm.uba.ar}

\begin{document}

\begin{abstract}
In this paper we revisit the notion of strict laura algebras through the lens of $\tau$-tilting theory to define the family of algebras determined by $\tau$-slices.
We show that the representation dimension of every algebra determined by $\tau$-slices satisfying mild conditions is at most three.
\end{abstract}

\maketitle

\section{Introduction} 
Let $A$ be a finite-dimensional algebra over an algebraically closed field $k$, let $\mod A$ be the category of finitely presented (right) $A$-modules and $\ind A$ the set of isomorphism classes of indecomposable $A$-modules.
The notions of left and right parts of a module category, which are denoted by $\L_A$ and $\R_A$ respectively, were introduced in \cite{HRS} to study and characterize quasi-tilted algebras.
Then they were used in \cite{AssemCoelho} to define the family of laura algebras, these are algebras that have only finitely many isomorphism classes of indecomposable objects that do not belong to the left or right part of its module category. 
Indeed, the name of this family can be explained by its definition: an algebra $A$ is laura if $\L_A\cup\R_A$ is cofinite in $\ind A$.

A laura algebra $A$ is said to be strict if $A$ is not quasi-tilted. 
If $A$ is a strict laura algebra then it was shown in \cite{ACT} that both $\L_A$ and $\R_A$ are functorially finite. 
Under these hypotheses, $\L_A$ is co-generated by a module $E$ which is an additive generator of the category of $\Ext$-injective objects in $\L_A$ and, dually, $\R_A$ is generated by a module $F$ which is an additive generator of $\Ext$-projective objects in $\R_A$.
Later, in \cite{Assem1} it was shown that $E$ and $F$ have the structure of a left section and a right section, respectively. 
These are modules with certain homological properties that induce some particular combinatorial structures in the Auslander-Reiten quiver $\Gamma(A)$ of $A$.
For a general overview on the representation theory of laura algebras we refer the reader to \cite{Surveylaura}.

The story of slices and sections in representation theory (at least in the sense that we mean here) starts with \cite{HappelRingel} where they introduced the notion of a slice module and showed that an algebra $A$ had a slice module $T$ if and only if $A$ is a tilted algebra, that is, the endomorphism algebra of a tilting module $T'$ over a hereditary algebra $H$. 
The sections were then introduced independently in \cite{Liu, Skowronski} where they gave a criterion to determine if an algebra is tilted or not depending on the existence of a complete section in its Auslander-Reiten quiver.
This notion was later generalized to the double sections where it was used to characterize double tilted algebras, also known as shod algebras. 
In the same spirit, the left and right sections were introduced and studied in \cite{Assem1}.
It is worth noting that a common theme among all these notions of sections is that all of them are related to a tilting module over a suitable tilted algebra. 

With the appearance of cluster algebras \cite{FZ}  started a program to categorify them using the representation theory of finite-dimensional algebras. 
In this context, a new family of algebras gained relevance: the family of cluster-tilted algebras \cite{BMR}. 
In a series of papers \cite{ABS1, ABS2, ABS3, ABS4} it was shown that the representation theory of cluster tilted algebras is deeply connected with the representation theory of tilted algebras. 
For example, 
it was shown in \cite{ABS2} that a cluster tilted algebra $A$ is the relation extension of a tilted algebra $C$ if and only if there is a local slice $T$ in $\mod A$ such that $C$ is isomorphic to $A/\Ann T$.

In parallel, the study of cluster algebras through representation theory unveiled a series of phenomena that instigated the study of several new topics that today are central in representation theory.
One of these topics is $\tau$-tilting theory, whose study was first axiomatized in \cite{AIR}, see also \cite{DF}.
It was a matter of time until a notion of slice was introduced in $\tau$-tilting theory.
This took place in \cite{Treffinger} where the definition of $\tau$-slice was given and the basic properties of these modules were studied. 

In this paper we define the family of algebras determined by $\tau$-slices based on the properties of the categories generated and cogenerated by two finite sets of $\tau$-slices $\{T_1, \cdots, T_s\}$ and $\{S_1, \cdots, S_t\}$, respectively, satisfying some technical conditions, see Definition~\ref{def:strictcycliclaura}.
In our first result, we show that this definition generalizes that of strict laura algebras. 

\begin{teo}[Theorem~\ref{strict laura is stric cyclic laura}]
Every strict laura algebra $A$ is determined by $\tau$-slices. 
\end{teo}

We also show that our definition is a proper generalization of strict laura algebras by giving an explicit example of an algebra determined by $\tau$-slices which is not strict laura, see Example~\ref{strict cyclic laura algebra not strict laura}. 
It is worth noting that the algebra $A$ of Example~\ref{strict cyclic laura algebra not strict laura} is a wild algebra having every indecomposable $A$-module lying in a cycle of $\mod A$. 
\medskip

The representation dimension, introduced in \cite{Au1}, serves as an invariant for finite-dimensional algebras, designed to measure their distance from being of finite representation type. 
Specifically, an algebra is semisimple precisely when its representation dimension is zero. 
No algebra has representation dimension equal to one. 
For non-semisimple algebras, a representation dimension of two indicates finite representation type, while a value of at least three signifies infinite type.

However, the full meaning and the properties of this invariant are far from understood except in the very lowest dimensions.
For example, a still unproven conjecture states that the representation dimension of an algebra of tame representation type is at most three. 
Thirty-two years after the introduction of this invariant, it was shown in \cite{Iyama} that it is always finite.
Later, it was discovered that algebras could have a representation dimension greater than three. 
For instance, in \cite{Ro1} it is proved that the representation dimension of the exterior algebra on a $d$-dimensional vector space is $d+1$. 
Other examples of algebras with large representation dimension were given in \cite{Bergh}, \cite{BerghOppermann1}, \cite{KrauseKussin}, \cite{Op1}, \cite{Op2}, \cite{Op3}, \cite{OppermannMiemietz}. 
Note, however, that the exact value of the representation dimension is known in only a few cases, and there is no method to calculate an effective upper bound.

Numerous significant classes of algebras have been demonstrated to possess a representation dimension of at most three (e.g., \cite{APT, APT2,AST, CP, EHI, Saorin, GT}). 
In this paper, we establish an upper bound for the representation dimension of algebras determined by $\tau$-slices satisfying mild conditions.

\begin{teo}[Theorem~\ref{teo principal}]
Let $A$ be an algebra determined by $\tau$-slices.
If $\bigcup_{j=1}^t \cogen S_j$ is contained in $\L_A$ or $\bigcup_{i=1}^s \Gen T_i$ is contained in $\R_A$ then 
$\repdim A \leq 3$. 
\end{teo}
\noindent A corollary, the finitistic dimension conjecture holds for algebras determined by $\tau$-slices satisfying the conditions of the previous theorem.
\medskip

The paper is organized as follows. 
In Section 2, we give the necessary preliminaries for the paper.
In Section 3, we define the family of algebras determined by $\tau$-slices. 
In the same section, we prove that every strict laura algebra is determined by $\tau$-slices. 
Moreover, we give an example of an algebra determined by $\tau$-slices which is not strict laura. 
Finally, in Section 4, we study the representation dimension of algebras determined by $\tau$-slices.

\section{Preliminaries}

\subsection{Background and notation.}
Throughout this paper $A$ is a finite-dimensional algebra over an algebraically closed field. 
For an algebra $A$, we denote by $\mod A$ the category
of finitely generated right $A$-modules and by $\ind A$ the subcategory of $\mod A$ having as objects a full set of representatives of the isomorphism classes of the indecomposable objects in $\mod A$. 

We denote by $\Gamma(A)$ the Auslander-Reiten quiver of $A$ and by $\tau$ the Auslander-Reiten translation in $\mod A$. 
For every module $M$ we denote by $|M|$ the number of isomorphism classes of indecomposable direct summands of $M$. 
Following \cite{AIR} we say that an $A$-module $M$ is $\tau$-rigid if $\Hom_A(M, \tau M)=0$. 
Moreover, we say that a $\tau$-rigid module $M$ is $\tau$-tilting if $|M|=|A|$, where $A$ is considered as a module over itself.

Let $M$ be an $A$-module. 
We denote by $\Ann M$ the annihilator ideal of $M$.
The full subcategory of $\mod A$ having as objects the direct sums of indecomposable summands of $M$ is denoted by $\add M$.
By $\Gen M$ (respectively $\cogen M$) we denote the full subcategory having as objects those modules $X$ such that there is an epimorphism $M_0 \rightarrow X $ (a monomorphism $X \rightarrow M_0$,
respectively), with $M_0\in \add M$. 

We denote the projective (or injective) dimension of a module $M$ by $\pd_A M$ (or $\id_A M$, respectively). 
Also, we denote by $\gldim A$ the global dimension of $A$.

Let $\mathcal{X}$ be an additive subcategory of $\mod A$. 
We denote by $\mathcal{X}^{\perp}$ the category of all modules $M$ such that $\Hom_A(-,M)\mid_{\mathcal{X}} = 0$. Dually, we denote by $^{\perp}\mathcal{X}$ the category of all modules $M$ such that $\Hom_A(M,-)\mid_{\mathcal{X}} = 0$.
A module $M$ in $\mathcal{X}$ is called Ext-projective in $\mathcal{X}$  if $\Ext^1_A(M,-)\mid_{\mathcal{X}} = 0$. Dually,
a module $N$ is called  Ext-injective in $\mathcal{X}$ if $\Ext^1_A(-, N)\mid_{\mathcal{X}} = 0$.

\subsection{Left and right part of the module category.}

Given $M, N\in \ind A$, a path from $M$ to $N$ in $\ind A$ is a sequence of nonzero morphisms $M = X_1 \rightarrow X_2 \rightarrow \cdots \rightarrow X_t = N $ $(t \geq 1)$  where $X_i \in \ind A$ for all $i$. In this case, we say that $M$ is a predecessor of $N$ and that $N$ is a successor of $M$.

We recall from \cite{HRS} that the right part $\R_A$ of $\mod A$ is the full subcategory of $\ind A$ defined by
$$\R_A = \{M \in \ind A \, | \, \id_A N \leq 1 \text{ for each successor }  N \text { of } M \}$$
Clearly, $\R_A$ is closed under successors. Dually, the left part,
$$\L_A= \{M \in \ind A \, | \, \pd_A N \leq 1 \text{ for each predecessor }  N \text { of } M \}$$
is a full subcategory of $\ind A$ closed under predecessors. 

\begin{defi}\label{def:laura}
    An algebra $A$ is called
\textit{laura} whenever $\L_A\cup \R_A$  is cofinite in $\ind A$, and it is a \textit{strict laura algebra} if it is laura but not quasi-tilted.
\end{defi}

Let $A$ be a strict laura algebra.
If $E$ is the direct sum of a complete set of representatives of the isomorphism classes of indecomposable Ext-injectives in the subcategory $\add \L_A$, then $\add \L_A$ coincides with the class $\cogen E$.
Dually, $\add \R_A$ coincides with the class $\Gen F$, where $F$ is the direct sum of a complete set of representatives of the isomorphism classes of indecomposable Ext-projectives in the subcategory $\add \R_A$. 
See \cite{Assem1}\\

\subsection{Sections and slices}

The definition of algebras determined by $\tau$-slices is based on the notion of $\tau$-slice introduced in \cite{Treffinger}. 
In this subsection, we recall several notions of slice modules and sections in Auslander-Reiten quivers that are instrumental to the rest of the paper. 
We start by recalling the notion of complete slices introduced in \cite{HappelRingel}.

\begin{defi}\cite[Definition 4.2.2]{HappelRingel}
A module $T$ is said to be a \textit{complete slice} if the following are satisfied.
\begin{enumerate}
    \item The isomorhism classes of indecomposable direct summands of $T$ form a full and connected subquiver $\Sigma_T$ in a connected component $\Gamma$ of $\Gamma(A)$.
    \item The module $T$ is sincere.
    \item The module $T$ is convex in $\mod A$.
    \item If $M$ is a  non-projective, indecomposable $A$-module, then at most one of M or $\tau M$ belongs to $\Sigma_T$.
    \item If $M$ and $S$ are indecomposable such that $S$ belongs to $\Sigma_T$ and there is an irreducible morphism $f : M \to S$, then either $M$ belongs to $\Sigma_T$ or $M$ is not injective and $\tau^{-1} M$ belongs to $\Sigma_T$.    
\end{enumerate}
\end{defi} 

\begin{teo}\cite[Theorem 7.2]{HappelRingel}
    An algebra $A$ is tilted if and only if there is a complete slice $T$ in $\mod A$. In this case $T$ is tilting and $\End_A(T)$ is hereditary.
\end{teo}

It is well-known that connected components of the Auslander-Reiten quiver of an algebra are deeply related with translation quivers.
In translation quivers one can speak about certain combinatorial objects known as sections and presections. 
The combinatorial properties of Auslander-Reiten quivers happen to determine algebraic nature of the objects under study. 
The first appeareance of sections in representation theory of finite-dimensional algebras came with the so-called Liu-Skowronsky criterion which is the following. 

\begin{defi}
    Let $(\Gamma, \tau)$ be a connected translation quiver. Then a finite, convex, acyclic subquiver $\Sigma$ of $\Gamma$ is said to be a \textit{complete section} if for every $x\in \Gamma$ there is a unique $n\in \mathbb{Z}$ such that $\tau^{n}x\in \Sigma$.
\end{defi}

\begin{teo}\cite{Liu, Skowronski}
    Let $A$ be an algebra. Then $A$ is tilted if and only if there is a tilting module $T$ in $\mod A$ such that the isomorphism classes of indecomposable direct summands of $T$ form a complete section in a connected component $\Gamma$ of the Auslander-Reiten quiver $\Gamma(A)$ of $A$.
\end{teo}

Later on, the notions of right and left section were introduced in \cite{Assem1} generalizing the previously known complete sections. 
We only give the definition of right sections, being dual the definition of a left section. 

\begin{defi}\cite[§ 2.1]{Assem1}
A full subquiver $\Sigma$ of a translation quiver $\Gamma$ is called a \textit{right section} if:
\begin{enumerate}
    \item $\Sigma$  is acyclic.
    \item For any $x\in \Sigma_0$ such that there exist $y \in \Gamma_0$ and a path from $x$ to $y$, there exists a unique $n \geq 0$ such that $\tau^{n}y$.
    \item $\Sigma$ is convex in $\mod A$.
\end{enumerate}
\end{defi}

The importance of left and right sections is based on the deep connection they have with tilted algebras, as shown in the following theorem.

\begin{teo}\cite[Theorem B]{Assem1}
    Let $A$ be an algebra and let $T$ be an $A$-module such that the isomorphism classes of indecomposable direct summands of $T$ determine a right section in a connected component $\Gamma$ of the Auslander-Reiten quiver $\Gamma(A)$ of $A$. 
    Then $A_\rho := A/\Ann T$ is a tilted algebra that has $T$ as a complete slice. Moreover, the category of successors of $T$ in $\mod A$ and in $\mod A_\rho$ coincide.
\end{teo}

With the arrival of $\tau$-tilting theory \cite{AIR}, the notions of slices and sections were studied in \cite{Treffinger}.
In order to define $\tau$-slices, we need to recall what a presection is in a translation quiver.

\begin{defi}
Let $\Gamma$ be a connected component of a translation quiver $(\Gamma, \tau)$.
A \textit{presection} $\Sigma$ in $\Gamma$ is an acyclic, full, connected subquiver of $\Gamma$ satisfying the following conditions.
\begin{enumerate}
    \item If $x \to y$ is an arrow in $\Gamma$ where $x$ is in $\Sigma$ then either $y$ or $\tau y$ is in $\Sigma$, but not both.
    \item If $x \to y$ is an arrow in $\Gamma$ where $y$ is in $\Sigma$ then either $x$ or $\tau^- x$ is in $\Sigma$, but not both.
\end{enumerate}
\end{defi}

\begin{defi}\cite[Definition~3.4]{Treffinger}
    Let $T$ be a $\tau$-rigid module. 
    We say that $T$ is a \textit{$\tau$-slice} in $\mod A$ if the full subquiver of $\Gamma(A)$ generated by the isomorphism classes of the indecomposable direct summands of $T$ forms a presection. A $\tau$-slice $T$ is said to be complete if $T$ is a sincere $A$-module.
\end{defi}

 There is also a deep connection between $\tau$-slices, tilted
algebras and complete slices, as shown in the following theorem.

\begin{teo}\cite[Theorem~3.14, Remark~3.15]{Treffinger}
    Let $T$ be a $\tau$-slice in $\mod A$ and let $\Ann T \subset A$.
    Then the algebra $A/\Ann T$ is a tilted algebra and $T$ is a complete slice in $\mod A/\Ann T$.
\end{teo}

It follows by \cite[Proposition 3.21]{Treffinger} that the notions of $\tau$-slices and local slices coincide
in cluster tilted algebras. Moreover, \cite[Proposition 4.3]{Treffinger} shows that complete $\tau$-slices and complete slices coincide for tilted algebras.

\section{Algebras determined by $\tau$-slices}

In this section, we introduce the algebras determined by $\tau$-slices and we show that every strict laura algebra satisfies this property.

\begin{defi}\label{def:strictcycliclaura}
An algebra $A$ is said to be \textit{determined by $\tau$-slices} if the following conditions are satisfied.

\begin{enumerate}
	\item[$(1)$] There are two finite sets of $\tau$-slice modules $\{T_1, \cdots, T_s\}$ and $\{S_1, \cdots, S_t\}$ verifying that $\Hom_A(T_i, S_j)=0$ for all $1 \leq i \leq s$, $1\leq j\leq t$.
    \item[$(2)$] The full additive subcategories 
    $$ \mod A \setminus \left(\bigcup_{i=1}^s \Gen T_i \right) \cup \left(\bigcup_{j=1}^t \cogen S_j \right) \text{ and }\left(\bigcup_{i=1}^s T_i ^\perp\right) \cap \left(\bigcup_{j=1}^t {}^\perp S_j\right)$$ 
    coincide. Moreover, this category contains only finitely many isomorphism classes of indecomposable modules. 
    We denote this subcategory of $\mod A$ by $\Y_A$. 
	\item[$(3)$] For every object $X$ in $\cup_{i=1}^s\Gen T_i$ there is a $k \in \{1, \cdots, s\}$ such that every morphism $f : L \to X$ where $L$ does not have any injective direct summand in $\Gen T_k$ and $L \in \add (A \oplus DA \oplus \cup_{i\neq k} T_i)$ admits a factorisation
	$$L \xrightarrow{} T' \xrightarrow{f'} X $$
	where $T'$ is in $\add T_k$.
    \item[$(3^{op})$] For every object $Y$ in $\cup_{j=1}^t\cogen S_j$ there is a $k \in \{1, \cdots, t\}$ such that every morphism $g : Y \to N$ where $N$ does not have any projective direct summand in $\cogen S_k$ and $N \in \add (A \oplus DA \oplus \cup_{j\neq k} S_j)$ admits a factorisation
	$$Y \xrightarrow{} S' \xrightarrow{f'} N $$
	where $S'$ is in $\add S_k$.
\end{enumerate}
\end{defi}

\begin{obs}
    It follows from the definition of algebras determined by $\tau$-slices that $\cup_{i=1}^s\Gen T_i$ is closed under successors and $\cup_{j=1}^t\cogen S_j$ is closed under predecessors.
    Indeed, let $f: M\to N$ be a nonzero map between indecomposable modules $M$ and $N$ and assume that $M$ is in $\cup_{i=1}^s\Gen T_i$.
    By condition $(1)$, $N$ cannot belong to $\cup_{j=1}^t\cogen S_j$.
    Then $N$ is either in $\cup_{i=1}^s\Gen T_i$ or in $\Y_A$.
    But $f\in \Hom_A(M,N)$, so $N\not\in \cup_{i=1}^sT_i^\perp$ and therefore $N\not\in\Y_A$ by condition $(2)$. 
    Hence $N\in \cup_{i=1}^s\Gen T_i$.
    A similar argument shows that $\cup_{j=1}^t\cogen S_j$ is closed under predecessors.
\end{obs}
 
The aim of this section is to show that our definition is a generalization of the notion of strict laura algebras. 
But before doing so, we recall some well-known properties of the module category over a laura algebra.
We state these results for the right part of the module category, but we note that the dual results also hold for the left part of the module category.

\begin{lema}\cite[Theorem B]{Assem1}\label{lemma: sigma es secci'on} Let $A$ be an algebra and $\mathcal{C}\subseteq \R_A$ be a full subcategory closed under successors having $\mathcal{F}$ as subcategory of Ext-projectives.
Let $\Gamma$ be a component of the Auslander-Reiten quiver $\Gamma(A)$ of $A$. 
If $\Sigma=\Gamma \cap \mathcal{F}\neq \emptyset$ then, $\Sigma$ is a right section of $\Gamma$ convex in $\ind A$. Moreover, $\Ann A/ \Sigma$ is a tilted algebra having $\Sigma$ as complete slice.
\end{lema}

\begin{lema}\cite[Theorems 4.2 and 5.1 ]{ACT}\label{A_rho}
Let $A$ be a strict laura algebra, then each connected component of its right support $A_\rho$ is a tilted algebra, and the restriction to this component of $F$ is a slice module.
\end{lema}

\begin{obs}\label{rmk:mismo tau}
Let $B$ be a quotient algebra of $A$, and $0 \rightarrow \tau_A X \rightarrow Y\rightarrow X \rightarrow 0$ be an almost-split sequence in $\mod A$, with both $X$ and $\tau_A X$  indecomposable $B$-modules. 
Then, since $\mod B$ is a full subcategory of $\mod A$, this sequence is also almost-split in $\mod B$. In particular, $\tau_BX=\tau_AX$.
\end{obs}

\begin{prop}\label{prop: F are tau-slices}
Let $A$ be a strict laura algebra,  let $F$ be the direct sum of a complete set of representatives of the isomorphism classes of indecomposable Ext-projectives in the subcategory $\add \R_A$ and consider $\Gamma$  be a component of the Auslander-Reiten quiver of $A$. If  $\Sigma =\Gamma \cap F\neq \emptyset$, then $\Sigma$ is a $\tau$-slice in $\ind A$.
\end{prop}

\begin{proof}
Follows from Lemma \ref{A_rho} that $\Sigma$ is a slice in $A_{\rho}$, therefore $\Sigma$ is $\tau$-rigid. Now, consider a morphism $X \rightarrow Y$ with $X\in \Sigma$. If $Y \in \Sigma$ there is nothing to show. If instead $Y \notin \Sigma$ since $\Sigma $ is a slice in $A_{\rho}$ and the morphism $X \rightarrow Y$ is in fact a morphism in $A_{\rho}$ then, $\tau_AY=\tau_{A_{\rho}}Y \in \Sigma$ by Remark \ref{rmk:mismo tau}. Finally, consider a morphism $X \rightarrow Y$ with $Y\in \Sigma$. Since $\R_A$ is closed under successors $\tau_A^{-1}X $ belongs to $\R_A$. If $X \in \Sigma$ there is nothing to prove. Assume, $X\notin \Sigma$. The fact that, by Lemma \ref{lemma: sigma es secci'on},  $\Sigma$ is a right section in $\Gamma$ implies that there exits $n\geq 0$ such that $\tau_A^n(\tau_A^{-1}X)\in \Sigma$. If $n>0$, then we have a path $\tau_A^{n-1}X \rightsquigarrow X\rightarrow Y$ and the convexity of $\Sigma$ implies that  $X\in \Sigma$. In consequence $n=0$ and $\tau_A^{-1}X\in \Sigma$.

\end{proof}

\begin{teo}\label{strict laura is stric cyclic laura}
Every strict laura algebra is determined by $\tau$-slices.
\end{teo}

\begin{proof}
Let $A$ be a strict laura algebra. 
We show that $A$ verifies the conditions of Definition~\ref{def:strictcycliclaura}.

First we verify condition $(1)$.
From \cite[Theorem A]{Assem1} and its dual, we know that there are modules $E$ and $F$ in $\mod A$ such that $\L_A = \cogen E$ and $\R_A = \Gen F$. 
Moreover it follows from Proposition \ref{prop: F are tau-slices} and its dual that $E$ and $F$ are isomorphic to the sum $E\cong \oplus_{i=1}^t S_i$ and $F\cong \oplus_{j=1}^s T_j$ where every $S_i$ and $T_j$ are $\tau$-slices.

Condition $(2)$ is a direct consequence of the definition of laura algebras, see Definition~\ref{def:laura}.

For condition $(3)$ we claim that  every module $X$ in $\Gen F$ and every morphism $f: L\rightarrow X$  where every indecomposable direct summand of $L$ does not belong to $\Gen F$ admits a factorization $L \rightarrow T' \rightarrow X$  where $T'$ is in $\add F$. 
It is enough to observe that from \cite[Theorem B]{Assem1}, $F$ is a slice module in $\mod A_{\rho}$, so that any morphism from a module in $\ind A \setminus \R_A$  to $X$  factors through $F$.
Finally condition $(3^{op})$ is shown by dual arguments.

\end{proof}

We now give an example of an algebra determined by $\tau$-slices which is not a strict laura algebra.

\begin{ejem}\label{strict cyclic laura algebra not strict laura}
Let $A$ be the  algebra  given by the following quiver:

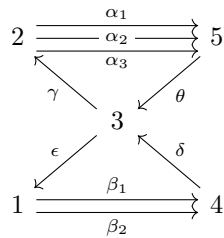
\[\begin{tikzcd}
	2 && 5 \\
	& 3 \\
	1 && 4
	\arrow["{\alpha_3}"', shift right=2, from=1-1, to=1-3]
	\arrow["{\alpha_1}", shift left=2, from=1-1, to=1-3]
	\arrow["{\alpha_2}"{description}, from=1-1, to=1-3]
	\arrow["\theta", shift left =.5, from=1-3, to=2-2]
	\arrow["\gamma", shift left =.5, from=2-2, to=1-1]
	\arrow["\epsilon"', shift right =.5, from=2-2, to=3-1]
	\arrow["{\beta_1}", shift left =1 , from=3-1, to=3-3]
	\arrow["{\beta_2}"', shift right=1, from=3-1, to=3-3]
	\arrow["\delta"', shift right =.5, from=3-3, to=2-2]
\end{tikzcd}\]
bounded by the relations $\gamma \alpha_i=0$, $\alpha_i\theta=0$,  $\epsilon \beta_j=0$ and $\beta_j\delta =0$ (for $i=1,2,3$ and $j=1,2$.
The Auslander-Reiten quiver $\Gamma(A)$ has a component of the following form. The indecomposables are represented by their Loewy series. 

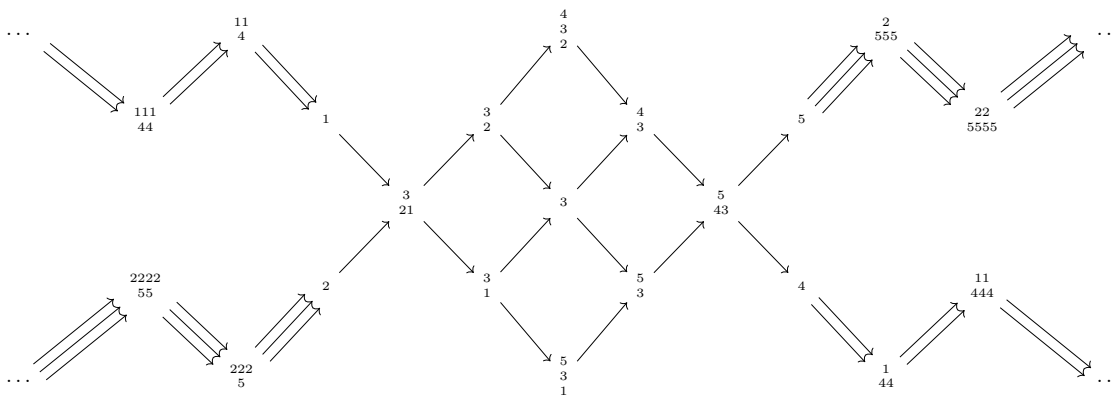
\begin{figure}[H]
\adjustbox{scale=.8,center}{
\begin{tikzcd}
	\ldots \hspace*{.5cm} && \rep{11\\4} &&&& \rep{4\\3\\2} &&&& \rep{2\\555} && \hspace*{.5cm} \ldots \\
	& \rep{111\\44} && \rep{1} && \rep{3\\2} && \rep{4\\3} && \rep{5} && \rep{22\\5555} \\
	&&&& \rep{3\\ 2 1 } && \rep{3} && \rep{5\\ 4 3} \\
	& \rep{2222 \\ 55} && \rep{2} && \rep{3\\1} && \rep{5\\3 } && \rep{4} && \rep{11\\ 444} \\
	\ldots  \hspace*{.5cm}  && \rep{222\\5} &&&& \rep{5\\3\\1} &&&& \rep{1\\ 44} && \hspace*{.5cm}  \ldots
	\arrow[from=1-1, to=2-2]
	\arrow[shift right=2, from=1-1, to=2-2]
	\arrow[shift left, from=1-3, to=2-4]
	\arrow[shift right, from=1-3, to=2-4]
	\arrow[from=1-7, to=2-8]
	\arrow[from=1-11, to=2-12]
	\arrow[shift right=2, from=1-11, to=2-12]
	\arrow[shift left=2, from=1-11, to=2-12]
	\arrow[shift left, from=2-2, to=1-3]
	\arrow[shift right, from=2-2, to=1-3]
	\arrow[from=2-4, to=3-5]
	\arrow[from=2-6, to=1-7]
	\arrow[from=2-6, to=3-7]
	\arrow[from=2-8, to=3-9]
	\arrow[from=2-10, to=1-11]
	\arrow[shift right=2, from=2-10, to=1-11]
	\arrow[shift left=2, from=2-10, to=1-11]
	\arrow[from=2-12, to=1-13]
	\arrow[shift right=2, from=2-12, to=1-13]
	\arrow[shift left=2, from=2-12, to=1-13]
	\arrow[from=3-5, to=2-6]
	\arrow[from=3-5, to=4-6]
	\arrow[from=3-7, to=2-8]
	\arrow[from=3-7, to=4-8]
	\arrow[from=3-9, to=2-10]
	\arrow[from=3-9, to=4-10]
	\arrow[shift left=2, from=4-2, to=5-3]
	\arrow[shift right=2, from=4-2, to=5-3]
	\arrow[from=4-2, to=5-3]
	\arrow[from=4-4, to=3-5]
	\arrow[from=4-6, to=3-7]
	\arrow[from=4-6, to=5-7]
	\arrow[from=4-8, to=3-9]
	\arrow[shift right, from=4-10, to=5-11]
	\arrow[shift left, from=4-10, to=5-11]
	\arrow[shift left, from=4-12, to=5-13]
	\arrow[shift right, from=4-12, to=5-13]
	\arrow[from=5-1, to=4-2]
	\arrow[shift right=2, from=5-1, to=4-2]
	\arrow[shift left=2, from=5-1, to=4-2]
	\arrow[shift right=2, from=5-3, to=4-4]
	\arrow[shift left=2, from=5-3, to=4-4]
	\arrow[from=5-3, to=4-4]
	\arrow[from=5-7, to=4-8]
	\arrow[shift right, from=5-11, to=4-12]
	\arrow[shift left, from=5-11, to=4-12]
\end{tikzcd}}
    \caption{A component of the Auslander-Reiten quiver of $A$}
    \label{fig:enter-label}
\end{figure}

An easy verification shows that 
$$T_1 =  \rep{11\\4} \oplus  \rep{1} \oplus \rep{3\\21} \oplus  \rep{2}\oplus  \rep{222\\5}, \quad
T_2 =  \rep{4\\3\\2} \oplus  \rep{3\\2} \oplus \rep{3\\21} \oplus  \rep{3\\1}\oplus  \rep{5\\3\\1}, \quad \text{ and } \quad
T_3 =  \rep{2\\555} \oplus  \rep{5} \oplus \rep{5\\43} \oplus  \rep{4}\oplus  \rep{1\\44}$$
are $\tau$-slices in $\mod A$.
Moreover, one can check that the sets $\{T_1, T_2, T_3\}$ and $\{\}$ are sets of $\tau$-slices that make $A$ an algebra determined by $\tau$-slices.

However, $A$ is not a laura algebra. 
Indeed, every indecomposable $A$-module is in a cycle containing one module of infinite projective and injective dimensions. 
Hence $\L_A$ and $\R_A$ contain no nonzero object. 
The claim follows from the fact that $A$ is of wild representation type.
\end{ejem}

\section{The representation dimension of algebras determined by $\tau$-slices}

Our next goal is to study the representation dimension of algebras determined by $\tau$-slices.
For this, we recall a series of preparatory results, most of which are well-known. 
See \cite[§1.2]{APT} and references therein.

\begin{lema}
Let $A$ be an algebra, $n$ be a positive integer, and $M$ be a generator-cogenerator of $\mod A$.
Then the global dimension of $\End_A (M)$ is at most $n$ if and only if for every object $X$ in $\mod A$ there is an exact sequence
$$
0 \xrightarrow{} M_{n-1} \xrightarrow{} \cdots \xrightarrow{} M_1 \xrightarrow{} M_0 \xrightarrow{f} X \xrightarrow{} 0
$$
where $M_i$ is in $\add M$ for every $0\leq i \leq n-1$ and the induced sequence
$$
0 \xrightarrow{} \Hom_A (Y, M_{n-1}) \xrightarrow{} \cdots \xrightarrow{} \Hom_A (Y, M_{0}) \xrightarrow{f^*} \Hom_A (Y, X) \xrightarrow{} 0
$$
is exact for every $Y$ in $\mod A$.
In particular, $\repdim A \leq n+1$.
\end{lema}

A generator-cogenerator $M$ in $\mod A$ satisfying the conditions of the above lemma is known as an Auslander generator for $\mod A$.

\begin{obs}\label{rmk:sec}
In this paper, we are particularly interested in the case $n=2$ in the previous lemma.
Under this extra hypothesis, to conclude that $\repdim A \leq 3$ it is enough to find for every $X$ in $\mod A$ a short exact sequence
$$0 \xrightarrow{} M_1 \xrightarrow{} M_0 \xrightarrow{f_X} X \xrightarrow{} 0$$
where $M_0$ and $M_1$ are in $\add M$ and $f_X : M_0 \to X$ is a right $\add M$-approximation.
It is worth noting that every approximation $f_X : M_0 \to X$ is an epimorphism since $M$ is a generator of $\mod A$.
Hence if $M \in \mod A$ is a generator-cogenerator of $\mod A$ such that for every $X$ in $\mod A$ there is a right $\add M$-approximation $f_X : M_0 \to X$ where $\ker f_X \in \add M$,  then $\repdim A \leq 3$.
\end{obs}

Based on the previous observation, it was shown in \cite{APT} that the representation dimension of all tilted algebras is at most three as a consequence of the following lemma. 

\begin{lema}\cite[Proposition~2.2]{APT}\label{lem:APT}
Let $A$ be a tilted algebra and let $T$ be a slice module in $\mod A$.
If $X$ is a module in $\Gen T$ and $f: M \to X$ is a right $\add (T \oplus DA)$-approximation of $X$ then $\ker f \in \add T$.
\end{lema}

\begin{teo}\cite[§2.3]{APT}
    Let $A$ be a tilted algebra and let $T$ be a slice module in $\mod A$. Then the module $M_T:=A\oplus DA \oplus T$ is an Auslander generator for $\mod A$. 
    Moreover, $gl. dim. (\End_A(M_T))\leq 3$.
\end{teo}

In a similar vein as in the results above, we show that under certain conditions, the representation dimension of algebras determined by $\tau$-slices is at most $3$.

\begin{teo}\label{teo principal}
Let $A$ be an algebra determined by $\tau$-slices.
If $\bigcup_{j=1}^t \cogen S_j$ is contained in $\L_A$ or $\bigcup_{i=1}^s \Gen T_i$ is contained in $\R_A$ then  $\repdim A \leq 3$ with Auslander generator
$$L:= A \oplus DA \oplus \left( \bigoplus_{i=1}^s T_i \right) \oplus \left( \bigoplus_{j=1}^t S_j \right) \oplus \left( \bigoplus_{Y \in \Y_A} Y  \right).$$
\end{teo}

\begin{proof}
We only show the statement under the assumption that $\bigcup_{j=1}^t \cogen S_j \subset \L_A$ since the other case follows from dual arguments.
Clearly $L$ is a generator-cogenerator of $\mod A$ since $A$ and $DA$ are direct summands of $L$.
Following Remark~\ref{rmk:sec} we need to find a right $\add L$-approximation $f_X: M_0 \to X$ where $\ker f_X \in \add L$.

Let $X$ be an $A$-module.
Since approximations are additive, without loss of generality,  we can assume that $X$ is indecomposable.
Since $A$ is an algebra determined by $\tau$-slices then $X$ either belongs to $\bigcup_{j=1}^t \cogen S_j$, to $\Y_A$ or to $\bigcup_{i=1}^s \Gen T_i$.

If $X \in \Y_A \cup \add \left( \bigoplus_{j=1}^t S_j\right)$ then it is easy to see that the identity map $1_X: X \to X$ is a right {$\add L$-approximation} satisfying the conditions of Remark~\ref{rmk:sec}.

Suppose that $X \in \left(\bigcup_{i=1}^s \Gen T_i\right)$ and take $g : M_X \to X$ a right $\add M$-approximation.
Since $A$ is an algebra determined by $\tau$-slices, there is a distinguished $k \in \{1, \dots, s\}$ such that $X \in \Gen T_k$.
Then we can write $M_X$ as $M_X = P'_X \oplus I_X \oplus I'_X \oplus Y_X \oplus T_X \oplus T'_X$ where:
\begin{itemize}
	\item $P'_X$ is the maximal projective direct summand of $M_X$ which is not in $\Gen T_k$;
	\item $I_X$  is the maximal injective direct summand of $M_X$ which is in $\Gen T_k$;
	\item $I'_X$  is the maximal injective direct summand of $M_X$ which is not in $\Gen T_k$;
	\item $Y'_X$ is the maximal direct summand of $M_X$ such that $Y_X$ is in $\add \Y_A$;
	\item $T_X$ is the maximal direct summand of $M_X$ such that $T'_X$ is in $\add T_k$;
	\item $T'_X$ is the maximal direct summand of $M_X$ such that $T''_X$ is in $ \add\{T_i \mid 1 \leq i \leq s\}\setminus \add T_k$.
\end{itemize}
Note that we do not need to consider the maximal projective direct summands of $M_X$ that is an object in $\Gen T_j$, since every such projective module is a direct summand of $T_j$ itself.
Rearranging the summands we set $N_X := I_X \oplus T_X$ and $N'_X :=  P'_X \oplus I'_X \oplus T'_X$ and we can write $g : M_X \to X$ as $[h, h'] : N_X \oplus N'_X \to X$.

The map $h': N'_X \to X$ factors through a map $s : S_X \to X$ where $S_X$ is in $\add T_j$ by \ref{def:strictcycliclaura}.$(3)$. 
Since $g : M_X \to X$ is a right $\add M$-approximation of $X$, it is easy to check that so is the map $[h, s] : N_X \oplus S_X \to X$.
Moreover, we have that $N_X \oplus S_X$ is an object in $\add (T_j \oplus DA)$ where the maximal injective direct summand of $N_X \oplus S_X$ is in $\Gen T_j$.
Now, it has been shown in \cite[Corollary~3.8]{Treffinger} that the algebra $A/ \Ann T_j$ is a tilted algebra having $T_j$ as a slice module.
Hence, we can apply Lemma~\ref{lem:APT} to conclude that $\ker [h,s]$ is in $\add T_j$.

The last case to consider is when $X$ is an indecomposable module of $\bigcup_{j=1}^t\cogen S_j \subset \L_A$ which is neither an object in $\Y_A$ nor in $\bigcup_{i=1}^s \Gen T_i$.
Then $\Hom_A (Y, X)=0$ for every $Y \in \Y_A$ because $\L_A$ is closed under predecessors. 
Similarly $\Hom_A (T, X)=0$ for every $T \in \bigcup_{i=1}^s \Gen T_i$ by Definition~\ref{def:strictcycliclaura}.$(1)$.

We claim that $\Hom_A(DA, X)=0$. 
Let $I$ be an indecomposable injective module in $\bigcup_{j=1}^t\cogen S_j$. Then $I\in \add(\bigcup_{j=1}^t S_j)$ since $\add(\bigcup_{j=1}^t S_j)$ is the category of $\Ext$-injective modules in the subcategory $\bigcup_{j=1}^t\cogen S_j$.
By hypothesis, $\bigcup_{j=1}^t\cogen S_j \subset \L_A$ and by the dual of Lemma.\ref{A_rho} we get that $\Hom_A(I, X)=0$ and the claim follows.

Since $X\in \L_A$ we have $\pd_A X \leq 1$.
Therefore the minimal projective cover $\pi_X: P^X_0 \to X$ is a right $\add L$-approximation of $X$ whose kernel belongs to $\add L$. 
This finishes the proof.
\end{proof}

We recall that the weak representation dimension $\wrepdim A$ of an artin algebra $A$ is the infimum of the global dimensions of the endomorphism algebras of the generators of $\mod A$. 
Clearly, $\wrepdim A \leq  \repdim A$. 
Thus, the next corollary follows immediately from our Theorem \ref{teo principal}.

\begin{coro}
Let $A$ be an algebra determined by $\tau$-slices, then $\wrepdim A \leq 3$.
\end{coro}

It was shown in \cite{IT} that, if an artin algebra $A$ verifies $\repdim A \leq 3$ then its finitistic dimension $\findim A$ is finite. 
We thus obtain the following corollary.

\begin{coro}
Let $A$ be an algebra determined by $\tau$-slices, then $\findim A < \infty$.
\end{coro}

\section*{Ackowledgements}
The authors thank Ibrahim Assem for the valuable comments and Marcelo Lanzilotta for suggesting a more appropriate name for the family of algebras we introduce in this paper.

The authors gratefully acknowledge financial support from CSIC (Comisión Sectorial de Investigación Científica) of Uruguay (grant no. 22520220100067UD).

HT acknowledges that this work was completed in spite of Argentina's current science policies, which severely hinder basic research in the country.

\end{document}